\documentclass[a4paper,12pt]{amsart}

\usepackage[utf8]{inputenc}
\usepackage{amsmath, amsthm, amsfonts, amssymb}
\usepackage[foot]{amsaddr}

\usepackage{todonotes}
\usepackage{hyperref} 
\usepackage[mathlines, pagewise]{lineno} 

 \newtheorem{theorem}{Theorem}[section]

 \newtheorem{remark}[theorem]{Remark}
 \numberwithin{equation}{section}

\usepackage{amsmath}
 \usepackage[mathscr]{euscript}

\title[The best $m$-term trigonometric approximations]{The best $\mathbf m$-term trigonometric approximations of the classes of periodic functions of one and many variables \\in the space $B_{q,1}$}

\author{K.V.~Pozharska$^{1, 2}$, A.S.~Romanyuk$^1$
}

\address{
$^1$Institute of Mathematics of the NAS of Ukraine, Kyiv, Ukraine
$^2$Faculty of Mathematics, Chemnitz University of Technology, 
Germany}

\email{pozharska.k@gmail.com}
\email{romanyuk@imath.kiev.ua}

\allowdisplaybreaks


\begin{document}

\begin{abstract}
Exact order estimates are obtained of the best $m$-term trigonometric approximations of the Nikol'skii-Besov classes $B^r_{p, \theta}$ of periodic functions of one and many variables in the space $B_{q,1}$. In the univariate case ($d=1$), we get the orders of the respective approximation characteristics on the classes $B^r_{p, \theta}$ as well as on the Sobolev classes $W^r_{p, {\boldsymbol{\alpha}}}$ in the space $B_{\infty,1}$ in the case $1\leq p \leq \infty$.
\end{abstract}

\maketitle

\setcounter{tocdepth}{1}


\setlength{\parskip}{\medskipamount}

  \section{Introduction.}
In the paper we investigate the best $m$-term trigonometric approximations of the Nikol'skii-Besov classes $B^r_{p, \theta}$ of periodic functions of one and many variables in the space $B_{q,1}$. The norm in this class is more strong than the corresponding $L_q$-norm.
 In the univariate case ($d=1$), we investigate a behaviour of the 
 mentioned approximation characteristics on the classes $B^{r}_{p, \theta}$ as well as on the Sobolev classes $W^{r}_{p, \alpha}$ in the space $B_{\infty,1}$.

The obtained results, from the one hand side generalize the respective statements for the Lebesque space $L_q$ \cite{Bel_87,Bel_88,Bel_98,DeVore_Teml_95,Dung_98,Dung_00,Rom_06,Rom_07,Stasyuk_14,Temlyakov_98}. From the other hand side, they complement the estimates of this quantity on the classes $B^r_{p, \theta}$ in the space $B_{q,1}$ \cite{Romanyuk_Romanyuk_Pozharska_Hembarska2023}.
We will comment on this more in remarks to the obtained results.

Let us further define the  classes $B^r_{p, \theta}$, the spaces  $B_{q,1}$  and the approximation characteristics.

Let $\mathbb{R}^d$, $d\geq 1$,  be the $d$-dimensional Euclidean space of points $\boldsymbol{x}=(x_1, \dots, x_d)$, and $L_p(\mathbb{T}^d)$,  
$1 \leq p \leq \infty$, $\mathbb{T}^d :=  \prod^{d}_{j=1}\,[ 0,\,2\pi)$,
be the space of functions  $f(\boldsymbol{x})$ that are
 $2\pi$-periodic in  each variable and
their norm
\begin{align*}
    \|f\|_p &:=    
   \Big((2\pi)^{-d}
\int_{\mathbb{T}^d}|f(\boldsymbol{x})|^p\,d{\boldsymbol{x}}\Big)^{1/p},\quad 1 \leq p < \infty,
\\
 \|f\|_\infty &:= 
 \operatorname{ess\,sup}\limits_{{\boldsymbol{x}}\in  {\mathbb{T}^d}}|f(\boldsymbol{x})|
  \end{align*}
  is finite.

Furthermore, let $k\in \mathbb{N}$ and $\boldsymbol{h}\in \mathbb{R}^d$.
For $f\in L_p(\mathbb{T}^d)$, we denote 
$\Delta_{\boldsymbol{h}} f(\boldsymbol{x}) = f(\boldsymbol{x}+\boldsymbol{h})- f(\boldsymbol{x})$, and define the difference of order $k$ with a step $\boldsymbol{h}$ by the following formula
$$
\Delta_{\boldsymbol{h}}^k f(\boldsymbol{x}) = \Delta_{\boldsymbol{h}}\Delta_{\boldsymbol{h}}^{k-1} f(\boldsymbol{x}), \quad
\Delta_{\boldsymbol{h}}^0 f(\boldsymbol{x}) = f(\boldsymbol{x}). 
$$

The modulus of smoothness of order $k$ for the function $f\in L_p(\mathbb{T}^d)$ is defined by the formula
$$
\omega_k(f,t)_p:= \sup_{|\boldsymbol{h}|\leq t}  \| \Delta_{\boldsymbol{h}}^k f \|_p, \ \text{where} \ |\boldsymbol{h}|=\sqrt{h_1^2+\dots+h_d^2}, \ t\in \mathbb{R}, \ t> 0.
$$

We say that the function $f\in L_p(\mathbb{T}^d)$ belongs to the space 
$B^r_{p, \theta}$, $1\leq p, \theta\leq\infty$, 
$r>0$, if
$$
\Bigg( 
\int_0^\infty (t^{-r} \omega_k(f,t)_p)^\theta \frac{{\rm d}t}{t}   
\Bigg)^{1/\theta} 
< \infty, \quad 1 \leq \theta < \infty
$$
and
$$
\sup_{t>0} t^{-r} \omega_k(f,t)_p< \infty, \quad \theta=\infty.
$$
The norm of the space $B^{r}_{p, \theta}$ is defined by the formulas
$$
 \|f\|_{B^{r}_{p, \theta}} :=  \|f\|_p + \Bigg( \int_0^\infty (t^{-r} \omega_k(f,t)_p)^\theta \frac{{\rm d}t}{t}   \Bigg)^{1/\theta}, \quad
 1 \leq \theta < \infty,
$$
and 
$$
 \|f\|_{B^{r}_{p, \infty}} := \|f\|_p + \sup_{t>0} t^{-r} \omega_k(f,t)_p, \quad
\theta=\infty,
$$
for some $k>r$.

The spaces $H^r_p \equiv B^{r}_{p, \infty}$ and $B^{r}_{p, \theta}$,
 $1 \leq \theta < \infty$, were introduced by S.M.~Nikol'skii \cite{Nikol'skii1951} and O.V. Besov \cite{Besov1961}, respectively. 
  
The Nikol'skii-Besov class is defined as a unit ball in the space $B^{r}_{p, \theta}$. We will use for it the same notation, i.e., put 
$$
B^r_{p,\theta} := 
\{f \in B^r_{p,\theta}\colon \ \|f\|_{B^r_{p,\theta}} \leq 1 \}.
$$

Note that the classes $B^{r}_{p, \theta}$ and $H^{r}_{p}$ were studied from the approximation viewpoint in \cite{Bel_87,Bel_98,Dung_98,Jiang_Liu_00,Rom_09,Rom_13,Romanyuk_Romanyuk2009,Romanyuk_Romanyuk2012}, where additional relevant references can be found. In what follows, it will be convenient for us to use the equivalent (up to absolute constants) definition for the norm in the space $B^{r}_{p, \theta}$.

Let   $V_l(t)$, $l \in \mathbb{N}$, $t\in \mathbb{R}$, denotes the  de la Vall\'ee-Poussin kernel  of  the  form
$$
V_l(t) := l^{-1} \sum\limits^{2l-1}_{k=l} D_k(t),
$$
where 
$$
D_k(t):= \sum_{m=-k}^k e^{imt}
$$
is the Dirichlet kernel.

The multidimensional kernel $V_l(\boldsymbol{x})$, $l \in \mathbb{N}$, $\boldsymbol{x}\in \mathbb{R}^d$, is defined by the formula
$$
V_l(\boldsymbol{x}):=\prod_{j=1}^d V_l(x_j).
$$

Let us put 
$$
 V_l(f,\boldsymbol{x}):= (f\ast V_l)(\boldsymbol{x}),
$$
i.e., define a
convolution of the function $f\in L_1(\mathbb{T}^d)$ with the multidimensional kernel 
$V_l(\boldsymbol{x})$.

For $f\in L_1(\mathbb{T}^d)$, we set 
\begin{align*}
    \sigma_0(f) & :=  \sigma_0(f, \boldsymbol{x}) = V_l(f, \boldsymbol{x}),
    \\
     \sigma_s(f) & :=  \sigma_s(f, \boldsymbol{x})
     = V_{2^s}(f, \boldsymbol{x})-V_{2^{s-1}}(f, \boldsymbol{x}), \ s\in \mathbb{N}.
\end{align*}

Then for the function $f\in B^{r}_{p, \theta}$, $1 \leq p,\theta \leq \infty$, $r>0$, the following equivalences hold \cite{Lizorkin_68}
\begin{align}
   \|f\|_{B^r_{p,\theta}} &\asymp  \left( \sum_{s\in \mathbb{Z}_+}\ 2^{sr\theta}  \| \sigma_s(f)  \|_p^\theta \right)^{1/\theta} , \quad 1 \leq \theta < \infty,
    \label{n1.1}
    \\
     \|f\|_{B^r_{p,\infty}} & \asymp \sup_{s\in \mathbb{Z}_+}  2^{sr} \| \sigma_s(f)  \|_p, \quad \theta=\infty,
     \label{n1}
\end{align}
where $\mathbb{Z}_+ = \mathbb{N}\cup \{0\}$. Here and below, for positive sequences $a(n)$ and $b(n)$, $n\in \mathbb{N}$, we will use the notation $a(n)\asymp b(n)$, that means that there exist the constants $0<C_1<C_2$ such that $C_1 a(n) \leq b(n) \leq C_2 a(n)$. If only $b(n) \leq C_2 a(n)$ (or $b(n) \geq C_1 a(n)$) holds, then we will write $b(n)\ll a(n)$ ($b(n \gg a(n))$).

Note, that in the case $1<p<\infty$ one can write an equivalent to (\ref{n1.1})-(\ref{n1}) norm definition, using some unions of ``blocks'' of the Fourier series of the function $f$ instead of $\sigma_s(f)$.
For $f\in L_p(\mathbb{T}^d)$, $1<p<\infty$, we define 
$$
f_{(0)}:= f_{(0)}(\boldsymbol{x}) = f(0), \ f_{(s)}:= f_{(s)}(\boldsymbol{x})= 
\sum_{\boldsymbol{k}\in \mu(s)}
 \widehat f(\boldsymbol{k}) e^{i(\boldsymbol{k}, \boldsymbol{x})}, \ s\in \mathbb{N},
$$
where
  $$
  \mu(s):= \{ \boldsymbol{k}\in \mathbb{Z}^d\colon \ 2^{s-1} \leq \max_{j=1,\dots, d} |k_j| < 2^s \},
  $$
$(\boldsymbol{k}, \boldsymbol{x})= k_1 x_1+\cdots+ k_d x_d$ and 
$\widehat f(\boldsymbol{k}) = (2\pi)^{-d} \int_{\mathbb{T}^d} f(\boldsymbol{t}) e^{-i(\boldsymbol{k}, \boldsymbol{t})} {\rm d} \boldsymbol{t}$
are the Fourier coefficients of $f$.

Then, if $1<p<\infty$, $1 \leq \theta \leq \infty$, $r>0$, for $f\in B^{r}_{p, \theta}$ we have
\begin{align*}
   \|f\|_{B^r_{p,\theta}} &\asymp  \left( \sum_{s\in \mathbb{Z}_+}\ 2^{sr\theta}  \| f_{(s)}  \|_p^\theta \right)^{1/\theta} , \quad 1 \leq \theta < \infty,
    \\
     \|f\|_{B^r_{p,\infty}} & \asymp \sup_{s\in \mathbb{Z}_+}  2^{sr} \| f_{(s)}  \|_p, \quad \theta=\infty.
    \end{align*}

Note, that for the spaces $B^r_{p,\theta}$ the following embeddings hold
$$
B^{r}_{p, 1} \subset B^{r}_{p, \theta_1} \subset B^{r}_{p, \theta_2} \subset B^{r}_{p, \infty} \equiv H^r_p,  \quad 1\leq \theta_1\leq \theta_2\leq \infty.
$$

Now we define a norm $\|\cdot\|_{B_{q,1}}$ in the subspaces $B_{q,1}$, $1 \leq q \leq \infty$, of functions $f\in L_q({\mathbb{T}}^d)$. Such norm is similar to the decomposition norm of functions from the Besov spaces $B^{r}_{p, \theta}$ (see \eqref{n1.1}). So, for trigonometric
polynomials $t$ with respect to the multiple trigonometric system  $\{ e^{i(\boldsymbol{k},\boldsymbol{x})}\}_{\boldsymbol{k} \in \mathbb{Z}^d}$, the norm $\|t\|_{B_{q,1}}$ is defined by the formula
$$
\|t\|_{B_{q,1}} := \sum_{s
\in \mathbb{Z}_+
} \|\sigma_{s}(t)\|_q.
$$
Analogically
we define the norm $\|f\|_{B_{q,1}}$
 for any function $f \in L_q(\mathbb{T}^d)$, such that the series
 $\sum_{s \in \mathbb{Z}_+}\|\sigma_{s}(f)\|_q$  is convergent.

Note, that in the case $1 < q < \infty$ it holds
$$
  \|f\|_{B_{q,1}} \asymp \sum\limits_{s \in \mathbb{Z}_+}\|f_{(s)}\|_q.  
$$

For $f \in B_{q,1}$, $1 \leq q \leq \infty$, the following relations hold:
\begin{equation}\label{n2}
    \|f\|_q \ll \|f\|_{B_{q,1}}; \qquad \|f\|_{B_{1,1}} \ll \|f\|_{B_{q,1}} \ll \|f\|_{B_{\infty, 1}}.
\end{equation}

Further we define the approximation characteristics that will be investigated in what follows and a close to it quantity that will appear in remarks.

Let  $\mathcal{X}$  be  a normed  space with the norm $\|\cdot\|_{\mathcal{X}}$ and $\Theta_m$ be a set of $m$ arbitrary $d$-dimensional vectors with integer coordinates. Let us denote by
$$
P(\Theta_m):= P(\Theta_m, \boldsymbol{x}) = 
\sum_{\boldsymbol{k}\in \Theta_m} c_{\boldsymbol{k}} e^{i(\boldsymbol{k}, \boldsymbol{x})}, \ c_{\boldsymbol{k}} \in \mathbb{C}, 
$$
a trigonometric polynomial with ``numbers'' of harmonics from the set $\Theta_m$. For $f\in \mathcal{X}$, we consider the quantity
$$
e_m(f)_{\mathcal{X}}:=\inf\limits_{c_{\boldsymbol{k}}}\inf\limits_{\Theta_m}
{\left\|f-  P(\Theta_m)\right\|_{\mathcal{X}}},
$$
which is called the best $m$-term
 trigonometric approximation of $f$. For a class
$F\subset {\mathcal{X}}$, we set  
$$
e_m(F)_{\mathcal{X}}:=\sup\limits_{f\in F}{e_m(f)_{\mathcal{X}}}.
$$

The quantity $e_m(f)_2:=e_m(f)_{L_2(\mathbb{T})}$ for univariate functions was introduced by S.B.~Stechkin \cite{Stechkin_1955} in order to formulate a criterion of absolute convergence of orthogonal series in the general case of approximations by polynomials with respect to arbitrary orthogonal system in a Hilbert space. 

Further the quantities $e_m(F)_{\mathcal{X}}$ for certain functional classes and spaces $\mathcal{X}=L_q(\mathbb{T}^d)$, $d\geq 1$, as well as other normed spaces, were studied by many authors. The detailed overview can be found in the papers \cite{Akishev_10,Basarkhanov_2014,Bel_87,Bel_88,Bel_98,DeVore_Teml_95,Dung_98,Dung_00,Hansen_Sickel_12,Kashin_Temlyakov1994,Rom_03,Rom_06,Rom_07,Romanyuk_Romanyuk_Pozharska_Hembarska2023,Shvai_16,Shydlich2009,Stasyuk_14,Temlyakov_98} and the monographs \cite{Dung_Temlyakov_Ullrich2019,Romanyuk2012,Temlyakov1993,Trigub_Belinsky2004}. 

As it was indicated above, when commenting the obtained results,  we will refer to the estimates of a close to  $e_m(F)_{\mathcal{X}}$ approximation characteristics.

For $f\in \mathcal{X}$, we denote
$$
S_{\Theta_m}(f):=S_{\Theta_m}(f,\boldsymbol{x})=
\sum_{\boldsymbol{k}\in \Theta_m}{\widehat{f}}(\boldsymbol{k} ){e^{i(\boldsymbol{k}, \boldsymbol{x})}} 
$$
and consider the quantity 
$$
e_m^\bot(f)_{\mathcal{X}}:=\inf\limits_{\Theta_m}{\|f-S_{\Theta_m}(f)\|}_{\mathcal{X}}.
  $$
  If $F\subset {\mathcal{X}}$ is a functional class, then we put  
$$
e_m^\bot(F)_{\mathcal{X}}:=\sup\limits_{f\in F}{e^\bot_m(f)_{\mathcal{X}}}.
$$

The quantity $e_m^\bot(F)_{\mathcal{X}}$ is called the best orthogonal trigonometric approximation of the class $F$ in the space $\mathcal{X}$. The quantities $e_m^\bot(F)_{\mathcal{X}}$  for different functional classes $F$ in the Lebesque space $L_q(\mathbb{T}^d)$ as well as in some of their subspaces were investigated in many papers (see, e.g., \cite{Fedunyk_Hembarska2022,Rom_02,Romanyuk_Romanyuk2020,Romanyuk_Yanchenko2022,Stepanyuk_2020}), where one can find a more detailed bibliography.

Note, that from the definitions of the quantities $e_m(F)_{\mathcal{X}}$ and $e_m^\bot(F)_{\mathcal{X}}$ we get the following relation
\begin{equation*}\label{n3}
    e_m(F)_{\mathcal{X}} \leq e_m^\bot(F)_{\mathcal{X}}.
\end{equation*}

\section{Estimates of the quantities $e_m(B^{r}_{p, \theta})_{B_{q,1}}$ ($d\geq 1$)}
In this part of the paper we prove (for the case $d\geq 1$) the exact order estimates of the best $m$-term trigonometric approximations of the classes $B^{r}_{p, \theta}$ in the space
$B_{q,1}$ in two cases: a) $1<p\leq 2 <q<\infty$; b) $2<p<q<\infty$. When considering the case a), the so-called ``small smoothness'' effect appeared, that was noted when investigating this quantity
in the $L_q$-space  on some classes of functions in the papers \cite{Rom_03,Stasyuk_14}.

In our case, it consists in the following. When the smoothness parameter $r$ crosses the limiting case $r=d/p$, a ``jump'' in the estimate of the quantity $e_m(B^{r}_{p, \theta})_{B_{q,1}}$ appears.
In other words, we showed that for $1<p\leq 2 <q<\infty$,  in the cases $d(1/p-1/q)<r<d/p$ and $r>d/p$ the quantities $e_m(B^{r}_{p, \theta})_{B_{q,1}}$ have different orders.

Before we move to formulating and proving the obtained results, let us recall the known statements that will be used.

{\bf Lemma A} \cite{Bel_88}. {\it
Let $2<q<\infty$. Then for any trigonometric polynomial 
$P\left(\Theta_n;\boldsymbol{x}\right)$ and any $m<n$, one can find a trigonometric polynomial $P\left(\Theta_m;\boldsymbol{x}\right)$ such that
$$
\left\|P(\Theta_n)-P(\Theta_m)\right\|_q\leq C_3(q) \sqrt{\frac{n}{m}}\left\|P\left(\Theta_n\right)\right\|_2,
$$
and, moreover, $\Theta_m\subset \Theta_n.$
}

{\bf Theorem A.}  {\it 
Let 
$$
t(\boldsymbol{x}) = \sum_{|k_j|\leq n_j} c_{\boldsymbol{k}} e^{i(\boldsymbol{k},\boldsymbol{x})},
$$
where $c_{\boldsymbol{k}}\in \mathbb{C}$, $k_j\in \mathbb{Z}$, $n_j\in \mathbb{N}$, $j=1,\dots d$. Then for $1\leq p<q\leq \infty$ the following inequality holds
}
\begin{equation}\label{n4}
 \| t\|_q \leq 2^d \prod_{j=1}^d n_j^{\frac{1}{p}- \frac{1}{q}} \| t\|_p . 
\end{equation}
The inequality \eqref{n4} was obtained by S.M. Nikol'skii \cite{Nikol'skii1951} and is referred to as the inequality for different metrics.

\begin{theorem}
\label{Th 1}
Let $d\geq 1$, $1<p\leq 2 <q<\infty$, $1 \leq \theta \leq \infty$. Then  the following relations hold
\begin{equation}\label{n5}
    e_m(B^{r}_{p, \theta})_{B_{q,1}}\asymp \begin{cases}
        m^{-\frac{q}{2}(\frac{r}{d}-\frac{1}{p}+\frac{1}{q})}, &  d(\frac{1}{p}-\frac{1}{q})<r<\frac{d}{p},\\
         m^{-\frac{r}{d}+\frac{1}{p}-\frac{1}{2}}, &  r>\frac{d}{p}.
    \end{cases}
\end{equation}
\end{theorem}
Note that in the case $1<p\leq 2 < q < p/(p-1)$, $r>d$, the order of the quantity $e_m(B^{r}_{p, \theta})_{B_{q,1}}$ was obtained in \cite{Romanyuk_Romanyuk_Pozharska_Hembarska2023}.
\begin{proof}[Proof of Theorem~\ref{Th 1}]
Let us first prove the upper estimates in \eqref{n5}. In view of the fact, that the respective estimates do not depend on the parameter $\theta$, it is sufficient to consider the case $\theta=\infty$, i.e., to prove the upper estimate of the quantity $e_m(H^{r}_{p})_{B_{q,1}}$.

Let $f\in H^{r}_{p}$ and a number $l\in \mathbb{N}$ is such that satisfies the condition $2^{dl} \leq m <2^{d(l+1)}$, i.e., $m\asymp 2^{dl}$. We will approximate the function $f$ by the polynomials of the form
\begin{equation}\label{n6}
P(\Theta_m, \boldsymbol{x}) = \sum_{s=0}^{l-1} f_{(s)}(\boldsymbol{x}) + \sum_{s=l}^{[\gamma l]} P(\Theta_{m_s}, \boldsymbol{x}),
  \end{equation}
  where the polynomials $P(\Theta_{m_s}, \boldsymbol{x})$ approximate the ``blocks'' $f_{(s)}(\boldsymbol{x})$ according to Lemma A, and the numbers $m_s$, $\gamma$ are chosen, depending on the considered case, so that the number of harmonics of the polynomial $P(\Theta_m, \boldsymbol{x})$ is of order $m$. (Here and in what follows,  $[a]$ denotes the integer part of the number $a$.)

  Hence, noting that
  $$
  f(\boldsymbol{x}) = \sum_{s\in \mathbb{Z}_+} f_{(s)}(\boldsymbol{x}),
  $$
and for the function $f\in H^{r}_{p}$ it holds
\begin{equation}\label{n7}
 \|f_{(s)}\|_p \ll 2^{-sr},
  \end{equation}
  by \eqref{n2} we can write
  \begin{align}
      \|f- P(\Theta_m) \|_{B_{q,1}} &\ll     \Bigg\|  \sum_{s=l}^{[\gamma l]}( f_{(s)} - P(\Theta_{m_s}))  \Bigg\|_{B_{q,1}} +    \Bigg\|  \sum_{s=[\gamma l]+1}^\infty  f_{(s)} \Bigg \|_{B_{q,1}}
   \nonumber    \\
      & := I_1+I_2 .
      \label{n8}  
  \end{align}

  Let us first estimate the quantity $I_2$.

  Taking into account the definition of the norm $ \|\cdot \|_{B_{q,1}}$, by the inequality of different metrics \eqref{n4} and the estimate \eqref{n7}, we get
  \begin{align}
    I_2 & =       \Bigg\|  \sum_{s=[\gamma l]+1}^\infty  f_{(s)}  \Bigg\|_{B_{q,1}} \asymp \sum_{s=[\gamma l]+1}^\infty  \| f_{(s)}  \|_{q}
    \nonumber
      \\
      &  \ll \sum_{s=[\gamma l]+1}^\infty 2^{ds(\frac{1}{p}-\frac{1}{q})}  \| f_{(s)}  \|_{p}\ll 
      \sum_{s=[\gamma l]+1}^\infty
      2^{-s(r-d(\frac{1}{p}-\frac{1}{q}))}
      \nonumber
      \\
      &  \asymp
      2^{-\gamma l d(\frac{r}{d}-\frac{1}{p}+\frac{1}{q})}.
      \label{n9}  
  \end{align}
  Moving to estimation of the term $I_1$, we use Lemma A. Hence, we get
  \begin{align}
      I_1 &=  \Bigg\|  \sum_{s=l}^{[\gamma l]}( f_{(s)} - P(\Theta_{m_s}))  \Bigg\|_{B_{q,1}} 
      \leq \sum_{s=l}^{[\gamma l]} \| f_{(s)} - P(\Theta_{m_{s}})  \|_{q} 
    \nonumber
      \\
      &         \ll \sum_{s=l}^{[\gamma l]} \left(\frac{2^{sd}}{m_s}\right)^{1/2} \| f_{(s)} \|_{2} 
      \ll \sum_{s=l}^{[\gamma l]} \frac{2^{\frac{sd}{2}} 
      2^{sd(\frac{1}{p} -\frac{1}{2}) }  \| f_{(s)} \|_{p} }{m_s^{1/2}}
 \nonumber
      \\ &
    \ll 
      \sum_{s=l}^{[\gamma l]}  \frac{2^{-sd(\frac{r}{d}- \frac{1}{p})}}{m_s^{1/2}} .
     \label{n10}   
  \end{align}

  To further estimate the quantity 
  $I_1$,
  let us separately consider two cases.

  a) Let $d(1/p-1/q)<r<d/p$. For $s\in \mathbb{Z}_+$ let us
 set
\begin{align}
     m_s &= [ 2^{dl} 2^{s(\frac{d}{p}-r)} 2^{-\frac{ql}{2}(\frac{d}{p}-r) } ] +1,
   \label{n11}  \\
     \gamma &= \frac{q}{2} .
     \nonumber
\end{align}

  It is easy to verify that for such choice of the numbers $m_s$ and $\gamma$, the order of harmonics of the polynomial \eqref{n6} does not exceed $m$.
  We have
   \begin{align*}
     \sum_{s=0}^{l-1}    | \mu(s) | +   \sum_{s=l}^{[\frac{q}{2}l]}   m_s
 &    \ll 2^{dl} + \left(\frac{q}{2}-1 \right)l + 2^{dl} 2^{-\frac{ql}{2}(\frac{d}{p}-r) }
      \sum_{s=l}^{[\frac{q}{2}l]} 2^{s(\frac{d}{p}-r)}
      \\
      & \ll 2^{dl} +  \left(\frac{q}{2}-1 \right)l + 2^{dl} 2^{-\frac{ql}{2}(\frac{d}{p}-r) }
      2^{\frac{ql}{2}(\frac{d}{p}-r)} \ll 2^{dl} \asymp m.
   \end{align*}
 (Here and in what follows, by $|A|$ we denote the number of elements of the finite set $A$.)

 Therefore, substituting the values of $m_s$ from \eqref{n11} into
    \eqref{n10} and making elementary transformations, we obtain the estimate of the quantity $I_1$:
  \begin{align}
      I_1 &
           \ll \sum_{s=l}^{[\frac{q}{2}l]} 2^{-ds(\frac{r}{d}-\frac{1}{p})} 2^{-\frac{dl}{2}} 2^{-\frac{s}{2}(\frac{d}{p}-r)} 
      2^{\frac{ql}{4}(\frac{d}{p}-r)}
       \nonumber
      \\
      & = 2^{-\frac{dl}{2}}   2^{\frac{ql}{4}(\frac{d}{p}-r)}
       \sum_{s=l}^{[\frac{q}{2}l]} 2^{-\frac{ds}{2}(\frac{r}{d}-\frac{1}{p})}
       \ll 2^{-\frac{dl}{2}}   2^{\frac{ql}{4}(\frac{d}{p}-r)}
       2^{-\frac{dql}{4}(\frac{r}{d}-\frac{1}{p})}
 \nonumber       \\ 
 & = 2^{-\frac{dl}{2}-\frac{qrl}{2} + \frac{dql}{2p}}
       = 2^{-\frac{dql}{2}( \frac{r}{d}-\frac{1}{p}+ \frac{1}{q} )}
       \asymp m^{-\frac{q}{2}( \frac{r}{d}-\frac{1}{p}+ \frac{1}{q} )}.
\label{n12}
  \end{align}  

  Returning to estimation of the quantity $I_2$ (see \eqref{n9}), we obtain for $\gamma=q/2$ that
  \begin{equation}\label{n13}
I_2 \ll 2^{-\frac{dql}{2}( \frac{r}{d}-\frac{1}{p}+ \frac{1}{q} )}
       \asymp m^{-\frac{q}{2}( \frac{r}{d}-\frac{1}{p}+ \frac{1}{q} )}.
  \end{equation} 

  Combining the estimates \eqref{n12} and \eqref{n13} with \eqref{n8}, we get for $f\in H^r_p$:
  $$
   e_m(f)_{B_{q,1}}\ll  m^{-\frac{q}{2}( \frac{r}{d}-\frac{1}{p}+ \frac{1}{q} )},
  $$
that concludes the proof of the upper estimate for the quantity 
   $e_m(B^{r}_{p, \theta})_{B_{q,1}}$, $1\leq \theta\leq \infty$, in the case 
$d(1/p-1/q)<r<d/p$.

b) Let $r>d/p$. Here we set
 \begin{align*}
  m_s& = [ 2^{dl(\frac{r}{d}- \frac{1}{p} +1 )   } 2^{-s(r-\frac{d}{p})} ] +1,
\\
\gamma &= \frac{\frac{r}{d}- \frac{1}{p} +\frac{1}{2}}{ \frac{r}{d}-\frac{1}{p}+\frac{1}{q}}   \, .
 \end{align*}
In this case, the number of harmonics in the polynomial \eqref{n6} is estimated as follows.
  \begin{align*}
     \sum_{s=0}^{l-1}    | \mu(s) | +   \sum_{s=l}^{[\frac{q}{2}l]}   m_s
 &    \ll 2^{dl} + (\gamma-1)l + 2^{dl(\frac{r}{d}- \frac{1}{p} +1) }
      \sum_{s=l}^{[\gamma l]} 2^{-s(r-\frac{d}{p})}
      \\
      & \ll 2^{dl} + (\gamma-1)l + 2^{dl(\frac{r}{d}- \frac{1}{p} +1) }
      2^{-l(r-\frac{d}{p})}
   \\
   &= 2^{dl} + 2^{dl} +(\gamma-1)l 
       \ll 2^{dl} \asymp m.
   \end{align*}
Further substituting the values of $m_s$ into
    \eqref{n10} and making  respective elementary transformations, we obtain
 \begin{align}
      I_1 &
      \ll \sum_{s=l}^{[\gamma l]} 2^{-ds(\frac{r}{d}-\frac{1}{p})} 2^{-\frac{dl}{2} (\frac{r}{d}- \frac{1}{p} +1)} 
      2^{\frac{ds}{2}(\frac{r}{d}-\frac{1}{p})} 
          \nonumber
      \\
      & = 2^{-\frac{dl}{2}  (\frac{r}{d}- \frac{1}{p} +1) } 
            \sum_{s=l}^{[\gamma l]}
            2^{-\frac{ds}{2}(\frac{r}{d}-\frac{1}{p})}
       \ll 2^{-\frac{dl}{2}  (\frac{r}{d}- \frac{1}{p} +1)  }  
      2^{-\frac{dl}{2}(\frac{r}{d}-\frac{1}{p})}
 \nonumber       \\ 
 & =  2^{-dl (\frac{r}{d}- \frac{1}{p} +\frac{1}{2})  }  
       \asymp m^{- \frac{r}{d}+\frac{1}{p}- \frac{1}{2}}.
\label{n14}
  \end{align}  
  
    What concerns the term $I_2$ from \eqref{n8}, here for the chosen value of $\gamma$ we get
  \begin{equation}\label{n15}
    I_2 \ll 2^{-dl (\frac{r}{d}- \frac{1}{p} +\frac{1}{2})  }  
       \asymp m^{- \frac{r}{d}+\frac{1}{p}- \frac{1}{2}}.
  \end{equation}

    Therefore, taking into account the estimates \eqref{n14}, \eqref{n15} and using \eqref{n8}, we obtain
    $$
    e_m(B^{r}_{p, \theta})_{B_{q,1}} \ll e_m(H^{r}_{p})_{B_{q,1}} \ll 
    m^{- \frac{r}{d}+\frac{1}{p}- \frac{1}{2}}.
    $$

    The upper estimates in \eqref{n5} are proved.

   Let us get the respective lower estimates. Here we note that they follow from the known estimates of the quantities 
   $e_m(B^{r}_{p, \theta})_{q} $.

 {\bf Theorem B.} {\it
 Let $1<p\leq 2 <q<\infty$, $1 \leq \theta \leq \infty$. Then for $d\geq 1$ the following relations hold
\begin{equation}\label{n16}
    e_m(B^{r}_{p, \theta})_{q}\asymp \begin{cases}
        m^{-\frac{q}{2}(\frac{r}{d}-\frac{1}{p}+\frac{1}{q})}, &  d(\frac{1}{p}-\frac{1}{q})<r<\frac{d}{p},\\
         m^{-\frac{r}{d}+\frac{1}{p}-\frac{1}{2}}, &  r>\frac{d}{p}.
    \end{cases}
\end{equation}
}
Note that in the case $d(1/p-1/q)<r<d/p$ this estimate was obtained in \cite{Stasyuk_14}, and for $r>d/p$ in \cite{DeVore_Teml_95}.
     
Hence, taking into account that $\| \cdot\|_{B_{q,1}} \gg  \| \cdot\|_q$
and using \eqref{n16}, we derive to the required lower estimates of the quantity 
 $e_m(B^{r}_{p, \theta})_{B_{q,1}}$.

 Theorem \ref{Th 1} is proved.
\end{proof}

In the following statement, we consider the case $2< p<q<\infty$.

\begin{theorem}\label{Thm2}
    Let $d\geq 1$, $2< p<q<\infty$, $1 \leq \theta \leq \infty$, $r> d/2$. Then  the following estimate holds
\begin{equation}\label{n17}
    e_m(B^{r}_{p, \theta})_{B_{q,1}}\asymp 
        m^{-\frac{r}{d}}.
  \end{equation}
\end{theorem}
\begin{proof}[Proof of Theorem~\ref{Thm2}]
    The upper estimate follows from Theorem \ref{Th 1} as $p=2$, in view of the embedding
    $B^{r}_{p, \theta}\subset B^{r}_{2, \theta}$, $p\geq 2$. 

    The respective lower estimate in \eqref{n17} is a corollary from the known statement for the $L_q$-space.
    
 {\bf Theorem C \cite{DeVore_Teml_95}.} {\it
 Let $d\geq 1$, $2< p<q<\infty$, $1 \leq \theta \leq \infty$, $r>d/2$. Then  the following estimate holds
\begin{equation}\label{n18}
    e_m(B^{r}_{p, \theta})_{q}\asymp 
        m^{-\frac{r}{d}}.
  \end{equation}
 }

   Hence, by the relations \eqref{n18} and \eqref{n2}, we obtain the lower estimate for the quantity $e_m(B^{r}_{p, \theta})_{B_{q,1}}$.

   Theorem \ref{Thm2} is proved.
\end{proof}

Let us comment the results of Theorems \ref{Th 1}, \ref{Thm2}.

First, let us recall that in the paper \cite{Romanyuk_Romanyuk_Pozharska_Hembarska2023}, the authors obtain the estimates for the quantities $e_m(B^{r}_{p, \theta})_{B_{q,1}}$ for other relations between the parameters
$p$ and $q$, namely they proved the following statement.

{\bf Corollary A.} {\it
 Let $d\geq 1$ and either $1\leq p\leq q\leq 2$ or $1\leq q\leq p\leq \infty$. 
 Then for $r>d(1/p-1/q)_+$, $1 \leq \theta \leq \infty$    the estimate 
\begin{equation}\label{n19}
    e_m(B^{r}_{p, \theta})_{B_{q,1}}\asymp 
        m^{-\frac{r}{d}+ (\frac{1}{p}-\frac{1}{q})_+}
  \end{equation}
   holds, where $a_+=\max\{a, 0\}$.
 }
 
It is worth mentioning, that the estimates \eqref{n19} are realized by approximation of functions $f\in B^{r}_{p, \theta}$ by the trigonometric polynomials with the spectrum in cubic regions of the form
$$
t(\boldsymbol{x}) = \sum_{ \boldsymbol{k} \in C^d(2^n)} c_{\boldsymbol{k}} e^{i(\boldsymbol{k},\boldsymbol{x})},  \ c_{\boldsymbol{k}} \in \mathbb{C},
$$
where $C^d(2^n):=\{\boldsymbol{k} \in  \mathbb{Z}^d\colon  |k_j|< 2^n,  j=1, \dots d\}$
under the condition $m\asymp 2^{dn}$.

What concerns the relations between the parameters $p$ and $q$ that were considered in Theorems \ref{Th 1}, \ref{Thm2}, in these cases the mentioned polynomials do not realize the orders of the respective quantities 
$e_m(B^{r}_{p, \theta})_{B_{q,1}}$.

It is also worth to compare the obtained results with the  estimates for  the best orthogonal trigonometric approximations $e^\bot_m(B^{r}_{p, \theta})_{B_{q,1}}$.

Let us further formulate the respective statement.

{\bf Theorem D } \cite{Romanyuk_Romanyuk_Pozharska_Hembarska2023}. {\it
 Let  $d\geq 1$, $1\leq p, q,\theta \leq \infty$, $(p,q)\notin \{ (1,1), (\infty, \infty) \}$.   Then for $r>d(1/p-1/q)_+$  it holds
\begin{equation}\label{n20}
    e_m^\bot(B^{r}_{p, \theta})_{B_{q,1}}\asymp 
        m^{-\frac{r}{d}+ (\frac{1}{p}-\frac{1}{q})_+}.
  \end{equation}
   }

We note that the relations \eqref{n20} are obtained by the approximation of functions $f\in B^{r}_{p, \theta}$ by their cubic Fourier sums $S_n(f)$ of the form
$$
S_n(f):= S_n(f,\boldsymbol{x}) = \sum_{s=0}^{n-1} f_{(s)} (\boldsymbol{x})
$$
under the condition $m\asymp 2^{dn}$.

Therefore, comparing the results of Theorems \ref{Th 1}, \ref{Thm2} with the estimate \eqref{n20} for the respective values of the parameters $r, p, q$, we see that the quantities $e_m(B^{r}_{p, \theta})_{B_{q,1}}$ and $e_m^\bot(B^{r}_{p, \theta})_{B_{q,1}}$ differ in order.

From the other hand side, comparing the estimates from Theorem D and Corollary A, we see that in both of the cases  $1\leq p\leq q\leq 2$ and $1\leq q\leq p\leq \infty$, $(p,q)\notin \{ (1,1), (\infty, \infty) \}$, it holds
$$
  e_m^\bot(B^{r}_{p, \theta})_{B_{q,1}}\asymp   e_m(B^{r}_{p, \theta})_{B_{q,1}}.
$$

\begin{remark}
    The question on the order of the quantity  $e_m^\bot(B^{r}_{p, \theta})_{B_{q,1}}$ for the critical smoothness parameter $r=d/p$, that was not considered in Theorem \ref{Th 1}, remains open.
  \end{remark}

\section{Estimate of the quantity 
$e_m(B^{r}_{p, \theta})_{B_{\infty,1}}$  ($d=1$)}

As a continuation of investigations made in the previous section, in what follows we consider the univariate case and obtain  first an exact order estimate for the quantity  $e_m(B^{r}_{p, \theta})_{B_{\infty,1}}$. Afterwards, using this result and the known estimate of the quantity 
$e_m(W^{r}_{p, \alpha})_{\infty}$, we get the order of the best $m$-term trigonometric approximation of the Sobolev classes $W^{r}_{p, \alpha}$ in the space $B_{\infty,1}$.

Let us formulate the statement that will be used in further argumentation. 

Let $T_n(x)$, $n\in \mathbb{N}$, $x\in \mathbb{R}$, is a polynomial of the form
$$
T_n:= T_n(x)= \sum_{k=-n}^n  c_{k} e^{ikx},  \ c_{k} \in \mathbb{C} .
$$

{\bf Lemma B} \cite{Bel_98}. {\it 
For every $2\leq p< \infty$ and $1\leq m\leq 2n$ there exists a trigonometric polynomial $T(\Theta_m)$ with at most $m$ harmonics, such that
it holds
\begin{equation}\label{n21}
    \left\|T_n-T(\Theta_m)\right\|_\infty\leq C_4 \left( \frac{n}{m}
        \log\left(\frac{n}{m}+1\right) \right)^{1/p}
    \left\|T_n\right\|_p,
\end{equation}
and the spectrum $\Theta_m$ is in the segment $[-2n, 2n]$.
}

\begin{theorem}\label{Thm3}
    Let $d=1$, $1\leq p, \theta\leq  \infty$, $r>\max\{1/p, 1/2\}$. Then the following estimate holds
    \begin{equation}\label{n22}
      e_m(B^{r}_{p, \theta})_{B_{\infty, 1}}\asymp 
        m^{-r+ (\frac{1}{p}-\frac{1}{2})_+}.
\end{equation}
\end{theorem}
\begin{proof}[Proof of Theorem~\ref{Thm3}]
    Let us get the upper estimate in (\ref{n22}) in the case $p=2$ and $\theta=\infty$, i.e., for the classes $H^r_2$. Let us choose $l\in \mathbb{N}$ such that
    the inequality $2^l\leq m <2^{l+1}$ holds and for $s\in \mathbb{Z}_+$ introduce the numbers
    \begin{align}
        m_s&=\begin{cases}
            2^s, & 0\leq s\leq l-1, \\
            [2^{l(r+\frac{1}{2})}  \cdot 2^{-s(r-\frac{1}{2})} ]+1, & l\leq s \leq [\gamma l], \\
            0, & s\geq [\gamma l]+1,
                  \end{cases}
          \label{n23}  \\
        \gamma & = \frac{r}{r-\frac{1}{2}}.
    \nonumber
    \end{align}
we will approximate a function $f\in H^r_2$ by the polynomial $T(\Theta_m, x)$ of the form
 \begin{equation}\label{n24}
 T(\Theta_m, x)=  \sum_{s=0}^{l-1} f_{(s)}(x) + \sum_{s=l}^{[\gamma l]} T(\Theta_{m_s}, x),
  \end{equation}
  where $T(\Theta_{m_s}, x)$ are the polynomials that approximate the ``blocks'' $f_{(s)}(x)$ of the Fourier series of the function $f(x)$ according to Lemma B.

 It is easy to verify that in this case, the number of harmonics  of the polynomial $ T(\Theta_m, x)$  for $r>1/2$ is of order $m$. Indeed,
  \begin{align*}
      \sum_{s=0}^{l-1} 2^s + \sum_{s=l}^{[\gamma l]} m_s &\ll 2^l + 2^{l(r+\frac{1}{2})}  \sum_{s=l}^{[\gamma l]}   2^{-s(r-\frac{1}{2})} + (\gamma-1)l 
      \\
      & \ll 2^l \cdot 2^{l(r+\frac{1}{2})} \cdot 2^{-l(r-\frac{1}{2})} + (\gamma-1)l \ll 2^l + (\gamma-1)l \ll  2^l  \asymp m.
  \end{align*}

Hence, for $f\in H^r_2$, by the choice of the polynomial (\ref{n24}), we can write
  \begin{align}
        \left\|f-T(\Theta_m)\right\|_{B_{\infty,1}} 
 &   \leq    \left\| \sum_{s=l}^{[\gamma l]} (f_{(s)} - T(\Theta_{m_{s}}) )\right\|_{B_{\infty,1}}  +  \left\| \sum_{s=[\gamma l]+1}^{\infty} f_{(s)}\right\|_{B_{\infty,1}} 
\nonumber \\
& := I_3+I_4.
\label{n25}
    \end{align}

We proceed with estimation of the term $I_3$.

Taking into account the definition of $\|\cdot\|_{B_{\infty,1}} $ and using a property of the convolution, we get
  \begin{align}
 I_3  &  =       \sum_{s=0}^\infty   \Big\|  \sigma_{s} \ast \sum_{s'=l}^{[\gamma l]}  (f_{(s')} - T(\Theta_{m_{s'}}) ) \Big\|_{\infty} 
 \leq  \sum_{s=l-1}^{[\gamma l]}   \Big\|  \sigma_{s} \ast \sum_{s'=s}^{s+1}  (f_{(s')} - T(\Theta_{m_{s'}}) ) \Big\|_{\infty}  
 \nonumber \\
& \leq  \sum_{s=l-1}^{[\gamma l]}   \|  \sigma_{s}\|_{1}  \cdot \Bigg\| \sum_{s'=s}^{s+1}  (f_{(s')} - T(\Theta_{m_{s'}}) ) \Bigg\|_{\infty} ,
    \label{n26}
    \end{align}
where $\sigma_s = V_{2^s}- V_{2^{s-1}}$.

Further, by the fact that $\|V_{2^s}\|_1\leq C_5$, $C_5>0$ (see, e.g., \cite[Ch. 1, $\S$1]{Temlyakov1993}), we have
 \begin{equation}\label{n27}
\|\sigma_s \|_1 = \| V_{2^s}- V_{2^{s-1}} \|_1 \leq 
\| V_{2^s}\|_1 + \|V_{2^{s-1}}\|_1 \leq C_6.
  \end{equation}
   Combining (\ref{n27}) with (\ref{n26}) and taking into account the choice (\ref{n23}) of the numbers $m_s$, we get
    \begin{align}
 I_3  & \ll      \sum_{s=l-1}^{[\gamma l]}   
 \Bigg\| \sum_{s'=s}^{s+1}  (f_{(s')} - T(\Theta_{m_{s'}}) ) \Bigg\|_{\infty} \leq  \sum_{s=l-1}^{[\gamma l]}  \sum_{s'=s}^{s+1} 
 \| f_{(s')} - T(\Theta_{m_{s'}}) \|_{\infty}
 \nonumber \\
& \leq  \sum_{s=l}^{[\gamma l]}  \| f_{(s)} - T(\Theta_{m_{s}}) \|_{\infty}
+  \| f_{([\gamma l]+1)} \|_{\infty} := I_5+ I_6.
    \label{n28}
    \end{align}

    To estimate the term $I_5$, we use the relation (\ref{n21}). In our notation it takes the form
    $$
    \| f_{(s)} - T(\Theta_{m_{s}}) \|_{\infty} \ll 
    \left( \frac{2^s}{m_s}
        \log\left(\frac{2^s}{m_s}+1\right) \right)^{1/2}  \| f_{(s)}\|_2 .
    $$

    Noting that for $f\in H^r_2$ the estimate $\| f_{(s)}\|_2 \ll 2^{-sr}$ holds and substituting the values of $m_s$, we obtain
     \begin{align}
 I_5  & \ll     2^{-\frac{l}{2}(r+\frac{1}{2})}  
  \sum_{s=l}^{[\gamma l]} 
 2^{\frac{s}{2}(r+\frac{1}{2})} \left(  
 \log\left(\frac{2^{s(r+\frac{1}{2})}}{2^{l(r+\frac{1}{2})}}+1\right)  \right)^{1/2} 2^{-sr}
  \nonumber \\
& \leq  2^{-\frac{l}{2}(r+\frac{1}{2})}  
  \sum_{s=l}^{[\gamma l]} 
 2^{-\frac{s}{2}(r-\frac{1}{2})} \left(  
 \log\left(\frac{2^{s(r+\frac{1}{2})}}{2^{l(r+\frac{1}{2})}}+1\right)  \right)^{1/2} 
  \nonumber \\
& \ll  2^{-\frac{l}{2}(r+\frac{1}{2})}   2^{-\frac{l}{2}(r-\frac{1}{2})}
= 2^{-lr} \asymp m^{-r}.
    \label{n29}
    \end{align}   

    Concerning the term $I_6$ we note that its estimate will be obtained during estimation of the quantity $I_4$. Namely, it will be shown that 
  \begin{equation}\label{n30}
I_6 \ll m^{-r}.
  \end{equation}  

  Hence, substituting (\ref{n29}) and (\ref{n30}) into (\ref{n28}) we get
    \begin{equation}\label{n31}
I_3 \ll m^{-r}.
  \end{equation} 

  To finish the upper estimation for the quantity   $e_m(H^{r}_{2})_{B_{\infty, 1}}$, it remains to estimate the term $I_4$.

 In view of the definition of  $ \|\cdot \|_{B_{q,1}}$, the inequality of different metrics \eqref{n4} and the values of $\gamma$, we can write
  \begin{align}
    I_4 & =       \Bigg\|  \sum_{s=[\gamma l]+1}^\infty  f_{(s)}  \Bigg\|_{B_{\infty,1}} \ll 
  \sum_{s=[\gamma l]+1}^\infty  \Bigg\|  \sigma_{s} \ast \sum_{s'=s}^{s+1}  f_{(s')}  \Bigg\|_{\infty} 
  \nonumber
      \\
      & \leq  \sum_{s=[\gamma l]+1}^\infty 
       \|\sigma_s \|_1  \Bigg\| \sum_{s'=s}^{s+1}  f_{(s')}  \Bigg\|_{\infty} \ll 
\sum_{s=[\gamma l]+1}^\infty \| f_{(s)} \|_{\infty} 
    \nonumber
      \\
      & \ll 
      \sum_{s=[\gamma l]+1}^\infty  2^{\frac{s}{2}}\| f_{(s)}  \|_{2} 
      \ll \sum_{s=[\gamma l]+1}^\infty 2^{-s(r-\frac{1}{2})} 
      \nonumber
      \\
      & \ll       2^{-l \gamma (r-\frac{1}{2})} = 2^{-lr} \asymp  m^{-r}.
      \label{n32}  
  \end{align}

  To conclude, we substitute (\ref{n31}) and (\ref{n32}) in (\ref{n25}) and get the upper estimate for the quantity 
  $e_m(B^{r}_{2, \theta})_{B_{\infty, 1}}$:
  \begin{equation}\label{n33}
e_m(B^{r}_{2, \theta})_{B_{\infty, 1}} \ll e_m(H^{r}_{2})_{B_{\infty, 1}} \ll m^{-r}.
  \end{equation} 

  From (\ref{n33}) we can easily get the upper estimate for the quantity 
 $e_m(B^{r}_{p, \theta})_{B_{\infty, 1}}$ for other values of the parameter $p$.

 Let us consider two cases.

 a) Let $2<p \leq \infty$. Then, in view of the embedding $B^{r}_{p, \theta} \subset B^{r}_{2, \theta}$, we can write
   \begin{equation*}
e_m(B^{r}_{p, \theta})_{B_{\infty, 1}} \ll e_m(B^{r}_{2, \theta})_{B_{\infty, 1}} \asymp m^{-r}.
  \end{equation*} 

  b) Let $1\leq p<2$. By the inequality of different metrics, for $f\in B^{r}_{p, \theta}$ we have
  \begin{align*}
 \| f\|_{B^{r}_{p, \theta}}   & \asymp \left(\sum_{s=0}^\infty 2^{sr\theta} \|\sigma_s(f)\|_p^\theta \right)^{1/ \theta} \gg
 \left(\sum_{s=0}^\infty 2^{sr\theta} 2^{s(\frac{1}{2}- \frac{1}{p})\theta}  \|\sigma_s(f)\|_2^\theta \right)^{1/ \theta}
       \\
      & =  \left(\sum_{s=0}^\infty 2^{s(r-\frac{1}{p}+ \frac{1}{2})\theta}  \|\sigma_s(f)\|_2^\theta \right)^{1/ \theta} 
      \asymp \| f\|_{B^{r-1/p+ 1/2}_{2, \theta}} \, .
  \end{align*}
  This yields the following embedding
  $$
  B^{r}_{p, \theta} \subset B^{r-1/p+ 1/2}_{2, \theta}, \ 1\leq p<2,
  $$
  and hence, by the estimate (\ref{n33}) with $r-1/p+ 1/2$ in place of $r$, we get
  $$
  e_m(B^{r}_{p, \theta})_{B_{\infty, 1}}  \ll   e_m(B^{r-1/p+ 1/2}_{2, \theta})_{B_{\infty, 1}} \ll m^{-r+\frac{1}{p}- \frac{1}{2}} .
  $$

  The upper estimate in Theorem \ref{Thm3} is proved.

  The lower estimate in (\ref{n22}) is a corollary from the known relation (see \cite{Rom_07})
     \begin{equation}\label{n35}
e_m(B^{r}_{p, \theta})_{\infty} \asymp m^{-r+(\frac{1}{p}- \frac{1}{2})_+}, \ 1\leq p,\theta \leq \infty, \ r>\max\left\{ 1/p, 1/2 \right\},
  \end{equation} 
  and the inequality $\|\cdot \|_{B_{\infty, 1}} \gg \|\cdot \|_{\infty}$.

  Theorem \ref{Thm3} is proved.
\end{proof}
\begin{remark}
    From the result of Theorem \ref{Thm3} and the estimate (\ref{n35}), the relation
    $$
     e_m(B^{r}_{p, \theta})_{B_{\infty, 1}} \asymp e_m(B^{r}_{p, \theta})_{\infty} \asymp m^{-r+(\frac{1}{p}- \frac{1}{2})_+}.
         $$
    follows.

    From the other hand side, comparing the results of  Theorem \ref{Thm3} and Theorem~D (for $d=1$), we notice that in the considered case the quantities $e_m(B^{r}_{p, \theta})_{B_{\infty, 1}}$ and 
    $e_m^\bot(B^{r}_{p, \theta})_{B_{\infty, 1}}$ differ in order.
\end{remark}

We conclude the paper by getting the analogical to Theorem \ref{Thm3}  statement for the Sobolev classes $W^r_{p, \alpha}$.  Further we define these classes (see, e.g., \cite[Ch. 1, $\S$3]{Temlyakov1993}).

Let 
$$
F_r(x,\alpha) = \sum_{k=1}^\infty k^{-r} \cos(kx-\frac{\pi \alpha}{2}), \ r>0, \  \alpha\in \mathbb{R} ,
$$
is the Bernoulli kernel.

Then by $W^r_{p, \alpha}$ we denote a class of functions $f$, such that can be written in the form
$$
f(x) = \frac{1}{2\pi} \int_{\mathbb{T}} \varphi(y) F_r(x-y,\alpha) {\rm d} y ,
$$
where $\varphi \in L_p$, $\|\varphi \|_p\leq 1$.

\begin{theorem}\label{Thm4}
     Let $d=1$, $1\leq p\leq  \infty$, $r>\max\{1/p, 1/2\}$. Then for $\alpha \in \mathbb{R}$ the following estimate holds
    \begin{equation}\label{n36}
      e_m(W^r_{p, \alpha})_{B_{\infty, 1}}\asymp 
        m^{-r+ (\frac{1}{p}-\frac{1}{2})_+}.
\end{equation}
\begin{proof}[Proof of Theorem~\ref{Thm4}]
    The upper estimate follows from Theorem \ref{Thm3}   as $\theta =\infty$ and the embedding $W^r_{p, \alpha} \subset H^r_p$ (see, e.g., \cite[Ch. 1, $\S$3]{Temlyakov1993}). So,
    $$
       e_m(W^r_{p, \alpha})_{B_{\infty, 1}} \ll    e_m(H^r_p)_{B_{\infty, 1}} \asymp 
        m^{-r+ (\frac{1}{p}-\frac{1}{2})_+}.
    $$

    The respective lower estimate in (\ref{n36}) follows from the relation (see \cite{Bel_98})
     \begin{equation}\label{n37}
      e_m(W^r_{p, \alpha})_{\infty}\asymp 
        m^{- \min\{r, r-\frac{1}{p}+\frac{1}{2}\}}, \ 1\leq p \leq \infty, \ r> 1/p,
\end{equation}
    and the inequality $\|\cdot \|_{B_{\infty, 1}} \gg \|\cdot \|_{\infty}$.

Theorem \ref{Thm4} is proved.
\end{proof}

\begin{remark}
    Comparing the estimates (\ref{n36}) and (\ref{n37}) we conclude that for $r>\max\{1/p, 1/2\}$ it holds
    $$
  e_m(W^r_{p, \alpha})_{B_{\infty, 1}}  \asymp   e_m(W^r_{p, \alpha})_{\infty} .
    $$
    
    From the other hand side, comparing the estimate  (\ref{n36}) with the result of \cite[Theorem 2]{Romanyuk_Romanyuk2020}
    $$
    e_m^\bot(W^r_{p, \alpha})_{B_{\infty, 1}} \asymp  
   m^{-r+\frac{1}{p}}, \ 1<p<\infty, \ r>1/p, \ \alpha\in \mathbb{R},
    $$
    we see that for respective values of the parameters $r$ and $p$ the quantities $  e_m(W^r_{p, \alpha})_{B_{\infty, 1}} $ and $e_m^\bot(W^r_{p, \alpha})_{B_{\infty, 1}}$ differ in order.
\end{remark}

\end{theorem}

\noindent\textbf{Acknowledgement.} KP would like to acknowledge support by the Philipp Schwartz Fellowship of the Alexander von Humboldt Foundation.

\end{document}